\UseRawInputEncoding
 \documentclass[11pt]{article}
\usepackage{pgfplots}
\usepackage{float}
\usepackage{dsfont}
\usepackage{amsmath}
\usepackage{amssymb}
\usepackage{algorithmic}
\usepackage{algorithm}
\usepackage{mathtools}
\usepackage{color}
\usepackage{array}
\usepackage{tikz}
\usetikzlibrary{arrows,
	arrows.meta,
	bending,calc,patterns,angles,quotes}
\usepackage{float}

\makeatletter
\providecommand*{\cupdot}{%
	\mathbin{%
		\mathpalette\@cupdot{}%
	}%
}
\newcommand*{\@cupdot}[2]{%
	\ooalign{%
		$\m@th#1\cup$\cr
		\hidewidth$\m@th#1\cdot$\hidewidth
	}%
}
\makeatother

\newtheorem{theorem}{Theorem}

\newtheorem{lemma}{Lemma}
\newtheorem{proposition}{Proposition}
\newtheorem{remark}{Remark}
\newtheorem{definition}{Definition}

\newtheorem{corollary}{Corollary}
\def\endpf{{\ \hfill\hbox{\vrule width1.0ex height1.0ex}\parfillskip 0pt
	}}
	\newenvironment{proof}{\noindent{\bf Proof:}}{\endpf}
\begin{document}
	\title{Recursive construction of a Nash equilibrium in a two-player nonzero-sum  stopping game with asymmetric information}
	\author{Royi Jacobovic  \thanks{The author is with the Korteweg-de Vries Institute for Mathematics, University of Amsterdam, Science Park 904, 1098 XH Amsterdam, the Netherlands.
			{\tt royi.jacobovic@mail.huji.ac.il.}} \thanks{Supported by the GIF Grant 1489-304.6/2019}}
	
	\date{\today}
	
	\maketitle

	\begin{abstract}
		We study a discrete-time finite-horizon  two-players nonzero-sum  stopping game where the filtration of Player 1 is richer than the filtration of Player 2. A major difficulty which is caused by the information asymmetry is that Player 2 may not know whether Player 1 has already stopped the game or not. Furthermore, the classical backward-induction approach is not applicable in the current setup. This is because when the informed player decides not to stop, he reveals information to the uninformed player and hence the decision of the uninformed player at time $t$ may not be determined by the play after time $t$, but also on the play before time $t$. In the current work we initially show that the expected utility of Player 2 will remain the same even if he knows whether Player 1 has already stopped. Then, this result is applied in order to prove that, under appropriate conditions, a recursive construction in the style of Hamad\'ene and Zhang (2010) converges to a pure-strategy Nash equilibrium.  
	\end{abstract}
	
	\section{Introduction}\label{sec: introduction}
	\subsection{Stopping games}
	Dynkin \cite{Dynkin1967} considered a game with two players who observe a stochastic process $(R_t,X_t)_{t=0}^\infty$. Player $1$ (resp. 2) picks a stopping time $\tau$ (resp. $\nu$) such that $P(X_{\tau}\leq0,X_{\nu}>0)=1$. The purpose of Player 1 (resp. 2) is to maximize (resp. minimize) the expected value of $R_{\min\{\tau,\nu\}}$. Dynkin proved that once $\sup_t R_t$ is integrable, the game has a value. Kifer \cite{Kiffer1971} and Neveu \cite{Neveu1975} considered a similar game in which two players observe a three-dimensional stochastic process $(X_t,Y_t,Z_t)_{t=0}^\infty$. Each player chooses a stopping time and the purpose of Player 1 (resp. 2) is to maximize (resp. minimize) the expected value of 
	\begin{equation}
	X_{\tau}\textbf{1}_{\{\tau<\nu\}}+Y_{\nu}\textbf{1}_{\{\tau>\nu\}}+Z_{\tau}\textbf{1}_{\{\tau=\nu\}}\,.
	\end{equation}
	In particular, it is possible that the players choose stopping times $\tau$ and $\nu$ such that $P(\tau=\nu)>0$. In two-player nonzero-sum stopping games, the processes $(X_t)$, $(Y_t)$ and $(Z_t)$ are two dimensional; the first (resp. second) coordinate is interpreted as the payoff of Player 1 (resp. Player 2). In general, the existence of an equilibrium point in these games is ensured when the joint distribution of the payoff processes satisfies some conditions. For example, see Morimoto \cite{Morimoto1986}, Mamer \cite{Mamer1987}, Ohtsubo \cite{Ohtsubo1987} and Hamad\'ene \textit{et al} \cite{Hamadene2010}. Also, some relevant literature reviews are given by  Sakaguchi \cite{Sakaguchi1995}, Novak \textit{et al.} \cite{Novak1999} and Ja\'skiewicz \textit{et al.} \cite{Jaskiewicz2018}.  
	
	Another strand of literature studies stopping games with random stopping times. See, e.g., Yasuda \cite{Yasuda1985}, Rosenberg \textit{et al.} \cite{Rosenberg 2001}, Shmaya \textit{et al.} \cite{Shmaya2004} for the discrete-time case and Laraki \textit{et al.} \cite{Laraki2013} for the continuous time case. Notably, there is an extensive literature about continuous time versions of stopping games. For example, consider the classical works of Bismut \cite{Bismut1977} and Lepeltier \textit{et al.} \cite{Lepeltier 1984}. More recent works in this direction are those of: Attard \cite{Attard2017,Attard2018}, De Angelis \textit{et al.} \cite{De Angelis2018a,De Angelis2018b}, Hamad\'ene \textit{et al.} \cite{Hamadene2010}. In addition, a continuous (resp. discrete) version of stopping games with more than two players was discussed by Hamad\'ene \textit{et al.} \cite{Hamadene2014b} (resp. Hamad\'ene \textit{et al.} \cite{Hamadene2014a}). 
	
	Stopping games are motivated by various applications. Traditionally, a task that has stimulated mathematicians to analyse models with unique implications on optimal-stopping theory is the celebrated secretary's problem (for surveys, see e.g., Ferguson \cite{Ferguson1989} and Freeman \cite{Freeman1983}). Some  versions of this problem with more than one decision maker were analyzed by Bruss \textit{et al.} \cite{Bruss1998}, Chen \textit{et al.} \cite{Chen1997}, Enns \textit{et al.} \cite{Enns1985}, Fushimi \cite{Fushimi1981},  and Sakaguchi \cite{Sakaguchi1980}. Some other applications of stopping games are:  Shrinking markets (see, e.g., Ghemawat \textit{et al.} \cite{Ghemawat}), duels (see, e.g., Lang \cite{Lang1975}) and financial markets with a recallable option (see, e.g., Hamad\'ene \textit{et al.} \cite{Hamadene2006,Hamadene2010} and Kifer \cite{Kifer2000}). 
	
	\subsection{Asymmetric information}
	Recently, there is a growing attention towards stopping games with asymmetric information. Lempa \textit{et al.} \cite{Lempa2013} considered a stopping game with a finite-horizon which is an exponentially distributed random variable, where only one player is exposed to the value of this random variable.  Esmaeeli \textit{et al.}    \cite{Esmaeeli2019} and Gr\"un \cite{Grun2013} presented two models in which the asymmetry of information is modelled as in the
	classical work of Aumann and Maschler \cite{Aumann1995}. Namely, there is a finite set of states of nature in which the game can take place. The true state is a random variable whose distribution is a common knowledge. The asymmetry in information is due to the assumption that the true state is exposed only to  Player 1. An extension of this framework was discussed by Renault \cite{Renault2006}. De Angelis \textit{et al.} \cite{De Angelis2021} considered a game with payoffs which are diffusion processes, where only one player knows the exact value of the drift and the diffusion coefficient. Mamer \cite{Mamer1987} and Ohtsubo \cite{Ohtsubo1991} considered a class of games with a monotone payoff structure where  players have information flows that are modelled as two possibly different filtrations. Gapaev \textit{et al.} \cite{Gapeev2021} considered a game with asymmetric information in the context of a financial market. In particular Gapaev \textit{et al.} \cite{Gapeev2021} refer to a solution concept which is called Stackelberg equilibrium (for an analysis of this solution concept in other stopping games see, e.g., the works of Ektr\"om \textit{et al.} \cite{Ekstrom2008} and Peshkir \cite{Peshkir2009}). Gensbittel \cite{Gensbittel2019} considered a two-player zero-sum stopping game with payoff processes which are determined by a Markov chain. The evolution of this chain is observed only by one of the Players.  Ashkenazi-Golan \textit{et al.} \cite{Ashkenazi-Golan2020} analyzed a two-player two-state Markov game such that only one of the players knows when the state changes.
	
	One major challenge in the analysis of stopping games with asymmetric information is due to the fact that they cannot be analysed by the classical backward-induction approach. To see why, notice that when an informed player decides not to stop, he reveals information to the uninformed players. Indeed, he reveals that his payoff by stopping in early stages is not too high. In particular, the decision of the uninformed player at time $t$ may not be determined by the play after time $t$, but also on the play before time $t$ (see also Skarupski \textit{et al.} \cite{Skarupski2018}). Evidently, as it was shown by Laraki \cite{Laraki2000}, a zero-sum stopping games with asymmetric information may not have a value even under the standard assumptions. However, there are also other cases in which an equilibrium result may be derived. For example, Laraki \textit{et al.} (see Section 2.8.4.3 of \cite{Laraki2015}) mentioned that for finite-horizon zero-sum stopping games with incomplete information on one side, the value exists in discrete-time using Sion's minmax theorem.
	
	\subsection{The current work}
	We study finite-horizon two-player nonzero-sum stopping games in discrete-time, where $Z\equiv X$: If the players stop simultaneously, the payoff is the same as if Player 1 stops alone. Similar assumptions appear in the works of Hamad\'ene \textit{et al.} \cite{Hamadene2010} and Ramsey \cite{Ramsey2007}. Formally,  consider an integer $0<T<\infty$  with some two discrete-time processes $(X_t^1,Y_t^1)_{t=0}^T$ and $(X_t^2,Y_t^2)_{t=0}^T$. We assume that Player 1 observes all four processes while Player 2 observes only $(X_t^2,Y_t^2)_{t=0}^T$. In addition, Player 1 (resp. 2) picks a stopping time $\tau$ (resp. $\nu)$ with respect to his information flow and his goal is to maximize the expectation of 
	\begin{equation}
	X^1_\tau\textbf{1}_{\left\{\tau\leq\nu\right\}}+Y^1_\nu\textbf{1}_{\left\{\tau>\nu\right\}}\ \ \left(\text{resp. }X^2_\nu\textbf{1}_{\left\{\nu<\tau\right\}}+Y^2_\tau\textbf{1}_{\left\{\tau\leq\nu\right\}}\right).
	\end{equation}
	One main difficulty in the analysis of this game is the fact that Player 1 may pick a stopping time $\tau$ which is not a stopping time with respect to the information flow of Player 2. Consequently, it is possible that Player 2 should decide whether to stop or keep playing without knowing whether Player 1 has already stopped. It is reasonable to expect that if Player 2 knew that Player 1 has not stopped yet, he could learn something new which might help him to gain greater utility. With some extent of surprise, we will show  that this intuition is wrong. Namely, Player 2 may always assume that Player 1 has not stopped yet and by doing so he will get the maximal expected utility that could be obtained under the assumption that he is informed about the stopping epoch of Player 1.  Then, this result is applied in order to prove that once $X_t^1\geq Y_t^1$ for every $0\leq t\leq T$, a recursive construction in the style of Hamad\'ene \textit{et al.} \cite{Hamadene2010} yields a pure-strategy Nash equilibrium. To see a crucial difference between the current work and Hamad\'ene \textit{et al.} \cite{Hamadene2010}, recall that Hamad\'ene \textit{et al.} \cite{Hamadene2010} studied stopping games with symmetric information. Thus, their construction yields stopping times which are measurable with respect to (w.r.t.) the information of Player 1 but not with respect to the information of Player 2. Hence, a proper modification of the construction is needed. 
	
	The rest of this paper is organized as follows: Section \ref{sec: description} contains a detailed presentation of the model assumptions with the precise statements of the main results. In addition, this section contains a sketch of the proof of the equilibrium result. This sketch is based on some benchmarks which are given with some intuitive explanations. For clarity, the technical proofs are supplied in Section \ref{sec: proof}.   
	
	\section{Game description and the main results}\label{sec: description}
	Let $T\in(0,\infty)$ be a positive integer and consider a probability space $(\Omega,\mathcal{F},P)$ with two filtrations $\mathbb{F}\equiv\left(\mathcal{F}_t\right)_{t=0}^T$ and $\mathbb{G}=\left(\mathcal{G}_t\right)_{t=0}^T$ such that $\mathcal{G}_t\subseteq\mathcal{F}_t$ for every $0\leq t\leq T$. Denote the set of all $\mathbb{F}$ (resp. $\mathbb{G}$)-stopping times by $\mathcal{T}(\mathbb{F})$ $\left(\text{resp. }\mathcal{T}(\mathbb{G})\right)$. Let $\mathcal{D}$ be the space of all $\mathbb{F}$-adapted $\mathbb{R}$-valued  processes $\left(S_t\right)_{t=0}^T$ such that 
	
	\begin{equation}\label{eq: integrability}
	E\left(\max_{0\leq t\leq T}|S_t|\right)<\infty\,.
	\end{equation}
	Let $X^1,X^2,Y^1,Y^2\in\mathcal{D}$ such that $X^2$ and $Y^2$ are adapted to $\mathbb{G}$. Consider a two-player stopping game $\Gamma$ in which: 
	\begin{itemize}
		\item Player 1 picks $\tau\in\mathcal{T}(\mathbb{F})$ and Player 2 picks $\nu\in\mathcal{T}(\mathbb{G})$.
		
		\item For every $\tau\in\mathcal{T}(\mathbb{F})$ and $\nu\in\mathcal{T}(\mathbb{G})$ the expected utility of Player 1 is given by
		\begin{equation}\label{eq: J_1}
		J_1(\tau,\nu)\equiv E\left[X^1_\tau\textbf{1}_{\left\{\tau\leq\nu\right\}}+Y^1_\nu\textbf{1}_{\left\{\tau>\nu\right\}}\right]\,,
		\end{equation}
		and the expected utility of Player 2 is given by
		\begin{equation}
		J_2(\tau,\nu)\equiv E\left[X^2_\nu\textbf{1}_{\left\{\nu<\tau\right\}}+Y^2_\tau\textbf{1}_{\left\{\tau\leq\nu\right\}}\right]\,.
		\end{equation}
	\end{itemize}
	
	\begin{remark}
		\normalfont Note that the assumption that $X^2$ and $Y^2$ are adapted to $\mathbb{G}$ practically means that Player 2 observes the evolution of his payoff processes over time.
	\end{remark}
	
	\begin{remark}
		\normalfont By the definitions of $J_1$ and $J_2$, Player 1 has a priority in stopping: The payoff if the two players stop simultaneously is the same as Player 1 stops alone.  
	\end{remark}
	
	\begin{definition}
		A pair $(\tau^*,\nu^*)\in\mathcal{T}(\mathbb{F})\times\mathcal{T}(\mathbb{G})$ constitutes a pure-strategy Nash equilibrium in $\Gamma$ if and only if (iff)
		\begin{equation}
		J_1(\tau,\nu^*)\leq J_1(\tau^*,\nu^*), \ \ \forall\tau\in\mathcal{T}(\mathbb{F})\,,
		\end{equation}
		and
		\begin{equation}
		J_2(\tau^*,\nu)\leq J_2(\tau^*,\nu^*), \ \ \forall\nu\in\mathcal{T}(\mathbb{G})\,.
		\end{equation}
	\end{definition}
	In the current work, we are going to prove that if 
	\begin{equation}\label{eq: sufficient condition}
	X_t^1\geq Y_t^1\ \ , \ \ \forall0\leq t\leq T\ , \ P\text{-a.s.},
	\end{equation}
	then $\Gamma$ admits a pure-strategy Nash equilibrium which follows from a construction in the style of Hamad\'ene \textit{et al.} \cite{Hamadene2010}.
	
	In the rest of this section, there is an elaboration which is necessary in order to make this statement more precise. In addition, we will review some important benchmarks like the analysis of the best response of Player 2 and the recursive construction which yields the equilibrium. Especially, the results which appear in the following Subsection 2.1 and Subsection 2.2 do not require the condition which appears in \eqref{eq: sufficient condition}.
	\subsection{The best response of Player 2}\label{subsec: best-response}
	For every $0\leq t\leq T$, let $\mathcal{T}_t(\mathbb{F})$ (resp. $\mathcal{T}_t(\mathbb{G})$)  be the set of all $\mathbb{F}$ (resp. $\mathbb{G}$)-stopping times $\tau$ such that $P(t\leq\tau\leq T)=1$. Furthermore, for every $\tau\in\mathcal{F}$ define the filtration $\mathbb{G}^\tau\equiv\left(\mathcal{G}_t^\tau\right)_{t=0}^T$ as follows:
	\begin{equation}
	\mathcal{G}_t^\tau\equiv\mathcal{G}_t\vee\sigma\left\{\textbf{1}_{\left\{\tau\leq s\right\}};0\leq s\leq t\right\}, \ \ \forall 0\leq t\leq T\,.
	\end{equation}
	Note that once Player 1 picks $\tau\in\mathcal{T}(\mathbb{F})$,  the only difference between Player 2 and an observer with a filtration $\mathbb{G}^\tau$ is the knowledge of whether Player 1 has already stopped. Thus, the following theorem states that even if Player 2 knows whether Player 1 has already stopped, his expected utility will remain the same. This is because at any moment $t$ such that $P(\tau>t)>0$, Player 2 (whose information flow is described by the filtration $\mathbb{G}$) may conduct his decision-making under the assumption that Player 1 has not stopped yet. 
	\begin{theorem}\label{theorem: conditional expectation}
		Let $\tau\in\mathcal{T}(\mathbb{F})$ and $t\in[0,T-1]$.  Then, there exists $\hat{\nu}_t\in\mathcal{T}_t(\mathbb{G})$ such that the next equality holds $P$-a.s.
		
		\begin{align}\label{eq: ess-sup}
		&E\left[X^2_{\hat{\nu}_t}\textbf{1}_{\left\{\hat{\nu}_t<\tau\right\}}+Y^2_{\tau}\textbf{1}_{\left\{\tau\leq\hat{\nu}_t\right\}}\big|\mathcal{G}_t^\tau\right]=\text{\normalfont ess}\sup_{\nu\in\mathcal{T}_t(\mathbb{G}^\tau)}E\left[X^2_\nu\textbf{1}_{\left\{\nu<\tau\right\}}+Y^2_{\tau}\textbf{1}_{\left\{\tau\leq\nu\right\}}\big|\mathcal{G}_t^\tau\right]\nonumber\\&=Y^2_\tau\textbf{1}_{\{\tau\leq t\}}+\textbf{1}_{\{\tau>t\}}\text{\normalfont ess}\sup_{\nu\in\mathcal{T}_t(\mathbb{G})}E_{P|\{\tau>t\}}\left[X^2_\nu\textbf{1}_{\{\nu<\tau\}}+Y^2_\tau\textbf{1}_{\{\tau\leq\nu\}}\big|\mathcal{G}_t\right]\,.
		\end{align}
	\end{theorem}
	The proof of Theorem \ref{theorem: conditional expectation} is supplied in Subsection \ref{subsec: theorem 1}.
	\subsection{Recursive construction}\label{subsec: recursive construction}
	The recursive construction is as follows: Let $\tau_1=\nu_1=T$ and for every $n\geq1$ define:
	\begin{align}
	&W^{1,n}_t\equiv\\&\text{ess}\sup_{\tau\in\mathcal{T}_t(\mathbb{F})}E\left[X^1_\tau\textbf{1}_{\left\{\tau<\nu_n\right\}}+\left(X^1_T\textbf{1}_{\{\nu_n=T\}}+Y^1_{\nu_n}\textbf{1}_{\{\nu_n<T\}}\right)\textbf{1}_{\left\{\tau\geq\nu_n\right\}}\big|\mathcal{F}_t\right] ,\ \ \forall t\in[0,T]\,,\nonumber\\& \widetilde{\tau}_{n+1}\equiv\inf\left\{t\geq0;W^{1,n}_t=X_t^1\right\}\wedge \nu_n\,,\\&\tau_{n+1}\equiv\begin{dcases}
	\widetilde{\tau}_{n+1}, & \widetilde{\tau}_{n+1}<\nu_n, \\
	\tau_n, & \text{o.w.}
	\end{dcases}
	\end{align}

	\begin{align}
	&W^{2,n}_t\equiv\text{ess}\sup_{\nu\in\mathcal{T}_t(\mathbb{G}^{\tau_n})}E\left[X^2_\nu\textbf{1}_{\left\{\nu<\tau_n\right\}}+Y^2_{\tau_n}\textbf{1}_{\left\{\tau_n\leq\nu\right\}}\big|\mathcal{G}_t^{\tau_n}\right]\ \ ,\ \ \forall t\in[0,T]\,,\\& \widetilde{\nu}_{n+1}\equiv\inf\left\{t\geq0;W^{2,n}_t=X^2_t\right\}\wedge \tau_n\,,\\&\nu_{n+1}\equiv\begin{dcases}
	\widetilde{\nu}_{n+1}, & \widetilde{\nu}_{n+1}<\tau_n, \\
	\nu_n, & \text{o.w.}
	\end{dcases}
	\end{align}
	By construction, both of the sequences $(\tau_n)_{n=1}^\infty$ and $(\nu_n)_{n=1}^\infty$ are nonincreasing and hence they have limits $\tau^*$ and $\nu^*$ as $n\to\infty$. When $\mathbb{F}$ and $\mathbb{G}$ are the same, Hamad\'ene \textit{et al.} \cite{Hamadene2010} showed (in an analogue continuous-time model) that $(\tau^*,\nu^*)$ is a pure-strategy Nash equilibrium. In our model, the following Proposition \ref{lemma: Hamadene} is an analogue statement. Its proof is provided in Subsection \ref{subsec: proposition 1}.
	
	\begin{proposition}\label{lemma: Hamadene}
		There exist limits 
		$\tau^*\equiv\lim_{n\to\infty}\tau_n$ and $\nu^*\equiv\lim_{n\to\infty}\nu_n$ such that:
		\begin{enumerate}
			\item $\tau^*\in\mathcal{T}(\mathbb{F})$ and $\nu^*\in\mathcal{T}(\mathbb{F})$.
			
			\item \begin{equation}
			J_1(\tau,\nu^*)\leq J_1(\tau^*,\nu^*), \ \ \forall\tau\in\mathcal{T}(\mathbb{F})\,.
			\end{equation}
			
			\item \begin{equation}
			J_2(\tau^*,\nu)\leq J_2(\tau^*,\nu^*), \ \ \forall\nu\in\mathcal{T}(\mathbb{G})\,.
			\end{equation}
		\end{enumerate}	
	\end{proposition}
	Importantly, $\nu^*$ is not necessarily a $\mathbb{G}$-stopping time, and hence an equilibrium result will not follow from the above-mentioned construction unless some modifications are applied. Specifically, for each $n\geq1$ and $0\leq t\leq T$ define
	\begin{align}
	&V^n_t\equiv\begin{dcases}
	\text{ess}\sup_{\theta\in\mathcal{T}_t(\mathbb{G})}E_{P|\left\{\tau_n>t\right\}}\left[X^2_\theta\textbf{1}_{\left\{\theta<\tau_n\right\}}+Y^2_{\tau_n}\textbf{1}_{\left\{\tau_n\leq\theta\right\}}|\mathcal{G}_t\right], & P\left(\tau_n>t\right)>0, \\
	\infty, & \text{o.w.},
	\end{dcases}
	\\& \theta_{n+1}\equiv \inf\left\{t\geq0;V^n_t=X^2_t\right\}\wedge T\,.
	\end{align}	
	In particular, $\theta^*\equiv\inf_{n\geq2}\theta_n$ is a $\mathbb{G}$-stopping time as an infimum of a sequence of $\mathbb{G}$-stopping times. Moreover, for each $n\geq1$, Theorem \ref{theorem: conditional expectation} implies that
	\begin{align}\label{eq: tilde nu}
	\widetilde{\nu}_{n+1}&=\inf\left\{t\geq0;W^{2,n}_t=X^2_t\textbf{1}_{\left\{t<\tau_n\right\}}+Y^2_{\tau_n}\textbf{1}_{\left\{\tau_n\leq t\right\}}\right\}\\&=\inf\left\{t\geq0;Y^2_{\tau_n}\textbf{1}_{\{\tau_n\leq t\}}+V_t^n\textbf{1}_{\left\{t<\tau_n\right\}}=X^2_t\textbf{1}_{\left\{t<\tau_n\right\}}+Y^2_{\tau_n}\textbf{1}_{\left\{\tau_n\leq t\right\}}\right\}\nonumber\\&=\inf\left\{t\geq0; V_t^n=X_t^2\right\}\wedge\tau_n\nonumber\\&=\theta_{n+1}\wedge\tau_n\ \ , \ \ P\text{-a.s.}\nonumber
	\end{align}
	and hence we shall deduce that
	\begin{equation}\label{eq: nu*}
	\nu^*=\inf_{n\geq1}\left\{\theta_{n+1};\theta_{n+1}<\tau_n\right\}\wedge T, \ \ P\text{-a.s.}
	\end{equation}	
	This identity yields the next Corollary \ref{cor: equivalence} which is going to be useful for the proof of the equilibrium result in the next subsection.   
	\begin{corollary}\label{cor: equivalence}
		\begin{equation}\label{eq: iff}
		\nu^*<\tau^*\Leftrightarrow \theta^*<\tau^*\ \ , \ \ P\text{-a.s.}
		\end{equation}
	\end{corollary}
	The proof for this corollary is given in Subsection \ref{subsec: corollary 1}.
	\subsection{Equilibrium result}\label{subsec: equilibrium}
	\begin{theorem}\label{thm: equilibrium}
		if the condition \eqref{eq: sufficient condition} is satisfied, then $(\tau^*,\theta^*)$ is a pure-strategy Nash equilibrium.
	\end{theorem}
	The following Lemma \ref{lemma: auxilary} states that it is enough to prove Theorem \ref{thm: equilibrium} under the assumption that $X_t^1>Y_t^1$, $P$-a.s. for every $0\leq t\leq T$. The proof for this Lemma appears in Subsection \ref{subsec: lemma 1}.
	
	\begin{lemma}\label{lemma: auxilary}
		Assume that the condition \eqref{eq: sufficient condition} is satisfied and let $\epsilon>0$. In addition, define a two-player stopping game $\Gamma_\epsilon$ which is identical to $\Gamma$ except that the expected utility of Player 1 is given by
		
		\begin{equation}
		J_{1,\epsilon}(\tau,\nu)\equiv E\left[\left(X^1_\tau+\epsilon\right)\textbf{1}_{\left\{\tau\leq\nu\right\}}+Y^1_\nu\textbf{1}_{\left\{\tau>\nu\right\}}\right], \ \ \forall\tau\in\mathcal{T}(\mathbb{F})\ , \ \nu\in\mathcal{T}(\mathbb{G})\,.
		\end{equation}
		If a pair $(\tau_0,\nu_0)\in\mathcal{T}(\mathbb{F})\times\mathcal{T}(\mathbb{G})$ is a pure-strategy Nash equilibrium in $\Gamma_\epsilon$, then it is also a pure-strategy Nash equilibrium in $\Gamma$.  
	\end{lemma}
	\subsubsection*{Proof of Theorem \ref{thm: equilibrium}} 
	Assume by contradiction that $P(\tau^*>\nu^*)>0$. Since $X_t^1>Y_t^1$ for every $t\in[0,T]$, then $J_1(\tau^*,\nu^*)<J_1(\nu^*\wedge\tau^*,\nu^*)$ and hence the second part of Proposition \ref{lemma: Hamadene} implies that $\nu^*\wedge\tau^*\notin\mathcal{T}(\mathbb{F})$. This is a contradiction with the first part of Proposition \ref{lemma: Hamadene} and hence $\tau^*\leq\nu^*$, $P$-a.s. As a result, Corollary \ref{cor: equivalence} and \eqref{eq: nu*} imply that $\tau^*\leq\theta^*\leq\nu^*$, $P$-a.s. and hence $J_2(\tau^*,\nu^*)=J_2(\tau^*,\theta^*)$. Thus, the third part of Proposition \ref{lemma: Hamadene} yields that  $\theta^*$ is a best response of Player 2 to $\tau^*$. 
	
	Assume by contradiction that there exists $\tau_0\in\mathcal{T}(\mathbb{F})$ such that $J_1(\tau_0,\theta^*)>J_1(\tau^*,\theta^*)$. Denote $\tau'\equiv\tau_0\wedge\theta^*$, and note that $\tau'\in\mathcal{T}(\mathbb{F})$. Since $X_t^1> Y_t^1$ for every $0\leq t\leq T$, deduce that
	\begin{equation}
	J_1(\tau',\theta^*)\geq J_1(\tau_0,\theta^*)>J_1(\tau^*,\theta^*)\,.
	\end{equation} 
	Now, $\tau'\leq\theta^*\leq\nu^*$, $P$-a.s., and hence it does not matter whether Player 2 picks either $\theta^*$ or $\nu^*$, when picking $\tau'$, Player 1 is the one who is responsible for stopping, i.e., $J_1(\tau',\theta^*)=J_1(\tau',\nu^*)$. On the other hand, we have already proved that $\tau^*\leq\theta^*\leq\nu^*$, $P$-a.s. and hence $J_1(\tau^*,\theta^*)=J_1(\tau^*,\nu^*)$. This leads to a contradiction  because $\tau^*$ is the best response of Player 1 to $\nu^*$.  $\blacksquare$
	
\section{Proofs of Theorem \ref{theorem: conditional expectation} and some auxiliary results }\label{sec: proof}

\subsection{Proof of Theorem \ref{theorem: conditional expectation}}\label{subsec: theorem 1}
Let $\tau\in\mathcal{T}(\mathbb{F})$ and $0\leq t\leq T$. 
The following lemma includes a characterization of the structure of a general event in $\mathcal{G}_t^\tau$. Intuitively speaking, it states that an event belongs to $\mathcal{G}_t^\tau$ if and only if conditioning on the available information about $\tau$ yields an event which belongs to $\mathcal{G}_t$.     
\begin{lemma}\label{lemma: G_t^tau}
	Let $\tau\in\mathcal{T}(\mathbb{F})$ and $0\leq t\leq T$. $A\in\mathcal{G}_t^\tau$ iff there exist $A^0\in\mathcal{G}_t$ and $\left\{A_s;0\leq s\leq t\right\}\subseteq\mathcal{G}_t$ for which 
	
	\begin{equation}\label{eq: event decomposition}
	A=\left(\{\tau>t\}\cap A^0\right)\cup\left[\cup_{s=0}^t\left(\{\tau=s\}\cap A_s\right)\right]\,.
	\end{equation}  
\end{lemma}
\begin{proof}
	If $A$ has the structure of \eqref{eq: event decomposition}, then immediately deduce that $A\in\mathcal{G}_t^\tau$ because $\{\tau>t\}\in\mathcal{G}_t^{\tau}$ and $\{\tau=s\}\in\mathcal{G}_t^\tau$ for every $s\in[0,t]$.  In order to prove the other direction,  observe that every 
	\begin{equation}
	A\in \mathcal{P}\equiv\left\{B\cap\{\tau\leq s\};B\in\mathcal{G}_t,0\leq s\leq t\right\}\cup \mathcal{G}_t
	\end{equation}
	has the structure of \eqref{eq: event decomposition}. Specifically, if $A=B\cap\{\tau\leq s\}$ for some $B\in\mathcal{G}_t$ and $s\in[0,t]$, then $A_u= B$ for every $u\in[0,s]$ and $A^0=A_{s+1}=\ldots=A_t=\emptyset$. Similarly, if $A\in\mathcal{G}_t$, then it is possible to set $A^0=A_1=\ldots=A_t=A$. In addition, it may be verified that $\mathcal{P}$ is a $\pi$-system. Now, let $\mathds{D}$ be the collection of all $A\in\mathcal{F}$ which have the structure of \eqref{eq: event decomposition} and observe that:
	\begin{enumerate}
		\item An insertion of $A^0=A_0=A_1=\ldots=A_t=\Omega$ implies that $\Omega\in \mathds{D}$.
		
		\item Assume that $A_1,A_2\in\mathds{D}$ are such that
		\begin{equation}
		A_i=\left(\left\{\tau>t\right\}\cap A_i^0\right)\cup\left[\cup_{s=0}^t\left(\left\{\tau=s\right\}\cap A_{s,i}\right)\right], \ \ i=1,2
		\end{equation}
		and $A_1\subseteq A_2$. As a result, deduce that \begin{equation}
		A_2\setminus A_1=\left(\left\{\tau>t\right\}\cap \widetilde{A}^0\right)\cup\left[\cup_{s=0}^t\left(\left\{\tau=s\right\}\cap \widetilde{A}_{s}\right)\right]
		\end{equation}
		where $\widetilde{A}^0\equiv A_2^0\setminus A_1^0\in\mathcal{G}_t$ and $\widetilde{A}_s\equiv A_{s,2}\setminus A_{s,1}\in\mathcal{G}_t$ for every $s\in[0,t]$. Therefore, deduce that $A_2\setminus A_1\in\mathds{D}$. 
		
		\item Assume that $A_1\subseteq A_2\subseteq\ldots$ is a sequence of events which belong to $\mathds{D}$ such that
		\begin{equation}
		A_i=\left(\left\{\tau>t\right\}\cap A_i^0\right)\cup\left[\cup_{s=0}^t\left(\left\{\tau=s\right\}\cap A_{s,i}\right)\right]\ \ , \ \ i\geq1\,.
		\end{equation}
		Then, 
		\begin{equation}
		\cup_{i\geq1}A_i=
		\left(\left\{\tau>t\right\}\cap \widetilde{A}^0\right)\cup\left[\cup_{s=0}^t\left(\left\{\tau=s\right\}\cap \widetilde{A}_{s}\right)\right]
		\end{equation}
		such that $\widetilde{A}^0\equiv\cup_{i\geq1}A^0_i\in\mathcal{G}_t$ and $\widetilde{A}_s\equiv\cup_{i\geq1}A_{s,i}\in\mathcal{G}_t$ for every $s\in[0,t]$. Therefore, deduce that $\cup_{i\geq1}A_i\in\mathds{D}$.
	\end{enumerate}   
	This means that $\mathds{D}$ is a Dynkin system such that $\mathcal{P}\subseteq\mathds{D}$ and hence Dynkin's $\pi-\lambda$ theorem yields that $\mathcal{G}_t^\tau=\sigma\left\{\mathcal{P}\right\}\subseteq \mathds{D}$ from which the result follows. 
\end{proof}
\newline\newline
Now we are ready to provide the proof of Theorem \ref{theorem: conditional expectation}. 
\subsubsection*{Proof of Theorem \ref{theorem: conditional expectation}}
For every $t\leq u\leq T-1$, let $\nu^*_u$ be the minimal-optimal stopping time which solves
\begin{equation}
\text{ess}\sup_{\nu\in\mathcal{T}_u(\mathbb{G}^\tau)}E\left[X^2_\nu\textbf{1}_{\left\{\nu<\tau\right\}}+Y^2_{\tau}\textbf{1}_{\left\{\tau\leq\nu\right\}}\big|\mathcal{G}_u^\tau\right]\,.
\end{equation}
Thus, Lemma \ref{lemma: G_t^tau} implies that for every $t\leq u\leq T-1$, there are $A_u^0\in\mathcal{G}_u$ and $\left(A_{s,u}\right)_{s=0}^u\subseteq\mathcal{G}_u$ such that
\begin{equation}\label{eq: structure}
\{\nu^*_u=u\}=\left(A_u^0\cap\{\tau>u\}\right)\cup\left[\cup_{s=0}^u\left(\{\tau=s\}\cap A_{s,u}\right)\right]\,.
\end{equation} 
As a result, deduce that
\begin{equation}\label{eq: u=t}
	\{\tau>t\}\cap\{\nu_t^*=t\}=\{\tau>t\}\cap A^0_t\,.
\end{equation}
Now, consider the case when $u>t$ and observe that for every $0\leq l\leq s$ such that $t\leq s\leq u-1$:
\begin{align}
	&\{\tau>u\}\subseteq\{\tau\neq l\}\,,\label{eq:l inclusion}\\&\{\tau>u\}\cap\{\tau\leq s\}=\emptyset\,.\label{eq:s emptyset}
\end{align}
Therefore, if we use the notation $\overline{A}\equiv\Omega\setminus A$ for every $A\subseteq\Omega$, then an insertion of \eqref{eq: structure} and applying standard algebra of sets imply that: 
\begin{align}\label{eq: master decomposition}
&\{\tau>u\}\cap\{\nu_t^*=u\}=\{\tau>u\}\cap\{\nu^*_u=u\}\cap\left(\cap_{s=t}^{u-1}\{\nu_s^*>s\}\right)\\&=\{\tau>u\}\cap\{\nu^*_u=u\}\cap\left(\cap_{s=t}^{u-1}\overline{\{\nu_s^*=s\}}\right)\nonumber\\&=\{\tau>u\}\cap A_u^0\cap\left[\cap_{s=t}^{u-1}\overline{\left[\left(A_s^0\cap\{\tau>s\}\right)\cup\left[\cup_{l=0}^s\left(\{\tau=l\}\cap A_{l,s}\right)\right]\right]}\right]\nonumber\\&=\{\tau>u\}\cap A_u^0\cap\left[\cap_{s=t}^{u-1}\left[\left(\overline{A_s^0}\cup\{\tau\leq s\}\right)\cap\left[\cap_{l=0}^s\left(\{\tau\neq l\}\cup \overline{A_{l,s}}\right)\right]\right]\right]\nonumber\\&=\{\tau>u\}\cap A_u^0\nonumber\\&\cap\left[\cap_{s=t}^{u-1}\left[\left[\overline{A_s^0}\cap\left[\cap_{l=0}^s\left(\{\tau\neq l\}\cup \overline{A_{l,s}}\right)\right]\right]\cup\left[\{\tau\leq s\}\cap\left[\cap_{l=0}^s\left(\{\tau\neq l\}\cup \overline{A_{l,s}}\right)\right]\right]\right]\right]\,.\nonumber
\end{align}
In particular, \eqref{eq:s emptyset} implies that for every $t\leq s\leq u-1$, 
\begin{equation}\label{eq: help1}
	\{\tau>u\}\cap\left[\{\tau\leq s\}\cap\left[\cap_{l=0}^s\left(\{\tau\neq l\}\cup \overline{A_{l,s}}\right)\right]\right]\subseteq\{\tau>u\}\cap\{\tau\leq s\}=\emptyset\,.
\end{equation}
In addition, \eqref{eq:l inclusion} implies that for every $0\leq l\leq s$ such that $t\leq s\leq u-1$, 
\begin{equation}
\{\tau>u\}\cap\left(\{\tau\neq l\}\cup\overline{A_{l,s}}\right)=\{\tau>u\}\,,	
\end{equation}
and hence 
\begin{equation}\label{eq: help2}
	\{\tau>u\}\cap\left[\overline{A_s^0}\cap\left[\cap_{l=0}^s\left(\{\tau\neq l\}\cup \overline{A_{l,s}}\right)\right]\right]=\{\tau>u\}\cap \overline{A_s^0}\,.
\end{equation}
Combining \eqref{eq: master decomposition},\eqref{eq: help1} and \eqref{eq: help2} all together yields that
\begin{equation}\label{eq: u>t}
	\{\tau>u\}\cap\{v_t^*=u\}=\{\tau>u\}\cap A^0_u\cap\left(\cap_{s=t}^{u-1}\overline{A^0_s}\right)\,.
\end{equation}
Therefore, \eqref{eq: u=t} and \eqref{eq: u>t} imply that for every $t\leq u\leq T-1$, we get
\begin{equation}
	\{\tau>u\}\cap\{v_t^*=u\}=\{\tau>u\}\cap K_u
\end{equation}
where
\begin{equation}
K_u\equiv A_u^0\setminus\cup_{s=t}^{u-1}A_s^0\,.
\end{equation}
In particular, observe that
\begin{equation}
K_u\in\mathcal{G}_u\ \ , \ \ \forall t\leq u\leq T-1
\end{equation} 
and 
\begin{equation}
K_{u_1}\cap K_{u_2}=\emptyset\ \ , \ \ \forall t\leq u_1<u_2\leq T-1\,.
\end{equation}
Therefore, it is possible to define a random variable $\hat{\nu}_t:\Omega\rightarrow[t,T-1]$ such that 
\begin{equation}
\{\hat{\nu}_t=u\}=K_u\ \ , \ \ \forall t\leq u\leq T-1
\end{equation}
and $\{\hat{\nu}_t=T\}=\overline{\{t\leq\hat{\nu}_t\leq T-1\}}$.
As a result, by construction we get that  $\hat{\nu}_t\in\mathcal{T}_t(\mathbb{G})$ and 		
\begin{equation}\label{eq: event equality}
\{\tau>u\}\cap\{\nu^*_t=u\}=\{\tau>u\}\cap\{\hat{\nu}_t=u\}\ \ , \ \ \forall t\leq u\leq T-1\,.
\end{equation}
In addition, the identity
\begin{align}
X^2_\nu\textbf{1}_{\left\{\nu<\tau\right\}}+Y^2_{\tau}\textbf{1}_{\left\{\tau\leq\nu\right\}}=\sum_{u=t}^{T-1}X^2_u\textbf{1}_{\{\tau>\nu=u\}}+Y^2_\tau\left(1-\sum_{u=t}^{T-1}\textbf{1}_{\{\tau>\nu=u\}}\right)
\end{align}
holds for every $\nu\in\mathcal{T}_t(\mathbb{G}^\tau)$ (especially, it holds for both $\nu^*_t$ and $\hat{\nu}_t$) and hence \eqref{eq: event equality} implies the first part of \eqref{eq: ess-sup}. 

For the proof of the second part of \eqref{eq: ess-sup}, consider some arbitrary $\nu\in\mathcal{T}_t\left(\mathbb{G}^\tau\right)$ and assume that $P(\tau>t)>0$.	Let $A\in\mathcal{G}_t$ and observe that:
\begin{align}\label{eq: conditional expectation11}
& E\left\{\textbf{1}_A\textbf{1}_{\{\tau>t\}}E_{P|\{\tau>t\}}\left[X^2_\nu\textbf{1}_{\{\nu<\tau\}}+Y^2_\tau\textbf{1}_{\{\tau\leq\nu\}}\big|\mathcal{G}_t\right]\right\}\\&=E\left\{\textbf{1}_{\{\tau>t\}}E_{P|\{\tau>t\}}\left[\left(X^2_\nu\textbf{1}_{\{\nu<\tau\}}+Y^2_\tau\textbf{1}_{\{\tau\leq\nu\}}\right)\textbf{1}_A\big|\mathcal{G}_t\right]\right\}\nonumber\\&=P\left(\tau>t\right)E_{P|\{\tau>t\}}\left\{E_{P|\{\tau>t\}}\left[\left(X^2_\nu\textbf{1}_{\{\nu<\tau\}}+Y^2_\tau\textbf{1}_{\{\tau\leq\nu\}}\right)\textbf{1}_A\big|\mathcal{G}_t\right]\right\}\nonumber\\&=P\left(\tau>t\right)E_{P|\{\tau>t\}}\left\{\left[X^2_\nu\textbf{1}_{\{\nu<\tau\}}+Y^2_\tau\textbf{1}_{\{\tau\leq\nu\}}\right]\textbf{1}_A\right\}\nonumber\\&=E\left\{\left[X^2_\nu\textbf{1}_{\{\nu<\tau\}}+Y^2_\tau\textbf{1}_{\{\tau\leq\nu\}}\right]\textbf{1}_A\textbf{1}_{\{\tau>t\}}\right\}\ \ , \ \ P\text{-a.s.}\nonumber
\end{align}
In addition,  $\nu\geq t$, $P$-a.s. and hence $P$-a.s., $\tau\leq t$ implies that $\tau\leq\nu$. This and \eqref{eq: conditional expectation11} yield that
\begin{align}\label{eq: conditional expectation}
&E\left\{\textbf{1}_A\left[Y^2_\tau\textbf{1}_{\{\tau\leq t\}}+\textbf{1}_{\{\tau>t\}}E_{P|\{\tau>t\}}\left(X^2_\nu\textbf{1}_{\{\nu<\tau\}}+Y^2_\tau\textbf{1}_{\{\tau\leq\nu\}}\big|\mathcal{G}_t\right)\right]\right\}\\&=E\left\{\textbf{1}_A\left[Y^2_\tau\textbf{1}_{\{\tau\leq t\}}\textbf{1}_{\{\tau\leq\nu\}}+\textbf{1}_{\{\tau>t\}}E_{P|\{\tau>t\}}\left(X^2_\nu\textbf{1}_{\{\nu<\tau\}}+Y^2_\tau\textbf{1}_{\{\tau\leq\nu\}}\big|\mathcal{G}_t\right)\right]\right\}\nonumber\\&=E\left\{\left[X^2_\nu\textbf{1}_{\{\nu<\tau\}}+Y^2_\tau\textbf{1}_{\{\tau\leq\nu\}}\right]\textbf{1}_A\right\}\,.\nonumber
\end{align} 
Note that in the last equality is valid because of the fact that $\nu\geq t$, $P$-a.s. and hence $\textbf{1}_{\{\nu<\tau\}}\textbf{1}_{\{\tau>t\}}=\textbf{1}_{\{\nu<\tau\}}$, $P$-a.s. Furthermore, notice that \eqref{eq: conditional expectation} holds for every 
\begin{equation}
A\in \mathcal{P}\equiv\left\{B\cap\{\tau\leq s\};B\in\mathcal{G}_t,0\leq s\leq t\right\}\cup \mathcal{G}_t\,.
\end{equation}
In addition, let $\mathds{D}$ be the set of all $A\in\mathcal{F}$ for which \eqref{eq: conditional expectation} holds. Then, observe that:

\begin{enumerate}
	\item $\Omega\in\mathcal{G}_t$ and hence \eqref{eq: conditional expectation} implies that $\Omega\in \mathds{D}$. 
	
	\item Assume that $A\in \mathds{D}$. Since $\textbf{1}_{\Omega\setminus A}=\textbf{1}_{\Omega}-\textbf{1}_A$ and $\Omega,A\in\mathds{D}$, then the linearity of conditional expectation yields that $\Omega\setminus A\in \mathds{D}$. 
	
	\item Assume that $\left(A_n\right)_{n=1}^\infty\subset \mathds{D}$ such that $A_i\cap A_j=\emptyset$ for every $i\neq j$. Then, $\textbf{1}_{\cup_{n=1}^\infty A_n}=\sum_{n=1}^\infty\textbf{1}_{A_n}$ and hence the linearity of conditional expectation yields that $\cup_{n=1}^\infty A_n\in \mathds{D}$. 
\end{enumerate}
Thus, $\mathds{D}$ is a Dynkin system. Since $\mathcal{P}$ is a $\pi$-system which is a subset of $\mathds{D}$, Dynkin's $\pi-\lambda$ theorem implies that $\mathcal{G}^{\tau}_t=\sigma\{\mathcal{P}\}\subseteq \mathds{D}$. As a result,  the definition of conditional expectation  implies that  
\begin{align}\label{eq: conditional expectations}
&E\left[X^2_\nu\textbf{1}_{\{\nu<\tau\}}+Y^2_\tau\textbf{1}_{\{\tau\leq\nu\}}\big|\mathcal{G}_t^\tau\right]=Y^2_\tau\textbf{1}_{\{\tau\leq t\}}\\&+\textbf{1}_{\{\tau>t\}}E_{P|\{\tau>t\}}\left[X^2_\nu\textbf{1}_{\{\nu<\tau\}}+Y^2_\tau\textbf{1}_{\{\tau\leq\nu\}}\big|\mathcal{G}_t\right]\ \ , \ \ P\text{-a.s.}\nonumber
\end{align}
Now, recall that it has already been proved  that $\hat{\nu}_t$ solves the essential supremum problem of the LHS of \eqref{eq: conditional expectations} over the domain $\mathcal{T}_t(\mathbb{G}^\tau)$. Thus, $P$-a.s. $\hat{\nu}_t$ solves the essential supremum problem of the RHS of \eqref{eq: conditional expectations} on the same domain, i.e., 
\begin{align}
&E_{P|\{\tau>t\}}\left[X^2_{\hat{\nu}_t}\textbf{1}_{\{\hat{\nu}_t<\tau\}}+Y^2_\tau\textbf{1}_{\{\tau\leq\hat{\nu}_t\}}\big|\mathcal{G}_t\right]\\&=\text{\normalfont ess}\sup_{\nu\in\mathcal{T}_t(\mathbb{G}^\tau)}E_{P|\{\tau>t\}}\left[X^2_\nu\textbf{1}_{\{\nu<\tau\}}+Y^2_\tau\textbf{1}_{\{\tau\leq\nu\}}\big|\mathcal{G}_t\right]\ \ , \ \ P\text{-a.s.}\nonumber
\end{align}
Consequently, since $\hat{\nu}_t\in\mathcal{T}_t(\mathbb{G})\subseteq\mathcal{T}_t(\mathbb{G}^\tau)$, deduce that 
\begin{align}\label{eq: redundant information}
&E_{P|\{\tau>t\}}\left[X^2_{\hat{\nu}_t}\textbf{1}_{\{\hat{\nu}_t<\tau\}}+Y^2_\tau\textbf{1}_{\{\tau\leq\hat{\nu}_t\}}\big|\mathcal{G}_t\right]\\&=\text{\normalfont ess}\sup_{\nu\in\mathcal{T}_t(\mathbb{G})}E_{P|\{\tau>t\}}\left[X^2_\nu\textbf{1}_{\{\nu<\tau\}}+Y^2_\tau\textbf{1}_{\{\tau\leq\nu\}}\big|\mathcal{G}_t\right]\ \ , \ \ P\text{-a.s.}\nonumber
\end{align}
This, with the first equality in \eqref{eq: ess-sup}, implies that the following equation holds $P$-a.s.
\begin{align}
&\text{\normalfont ess}\sup_{\nu\in\mathcal{T}_t(\mathbb{G}^\tau)}E\left[X^2_\nu\textbf{1}_{\{\nu<\tau\}}+Y^2_\tau\textbf{1}_{\{\tau\leq\nu\}}\big|\mathcal{G}_t^\tau\right]=Y^2_\tau\textbf{1}_{\{\tau\leq t\}}\\&+\textbf{1}_{\{\tau>t\}}\text{\normalfont ess}\sup_{\nu\in\mathcal{T}_t(\mathbb{G})}E_{P|\{\tau>t\}}\left[X^2_\nu\textbf{1}_{\{\nu<\tau\}}+Y^2_\tau\textbf{1}_{\{\tau\leq\nu\}}\big|\mathcal{G}_t\right]\nonumber
\end{align}
and the result follows. $\blacksquare$

\subsection{Proof of Proposition \ref{lemma: Hamadene}}\label{subsec: proposition 1}
Proposition \ref{lemma: Hamadene} follows from the same arguments which were made by Hamad\'ene \textit{et al.} \cite{Hamadene2010} with certain modifications. Now, we elaborate on these modifications via several lemmata. 
	\begin{lemma}
		For each $n\geq1
		$, $\tau_n\in\mathcal{T}(\mathbb{F})$ and $\nu_n\in\mathcal{T}(\mathbb{G}^\tau)$. In addition, both sequences $\left(\tau_n\right)_{n\geq1}$ and $\left(\nu_n\right)_{n\geq1}$ are $P$-a.s. nonincreasing.  
	\end{lemma}
	
	\begin{proof}
		For $n=1$, the claim is immediate. Assume that it holds for some $n\geq1$ and notice that this assumption yields that $\nu_n\in\mathcal{T}(\mathbb{F})$. Now,
		define
		
		\begin{equation}
		U^{1,n}_t\equiv X_t^1\textbf{1}_{\{t\leq\nu_n\}}+\left(X^1_T\textbf{1}_{\{\nu_n=T\}}+Y^1_{\nu_n}\textbf{1}_{\{\nu_n<T\}}\right)\textbf{1}_{\left\{t\geq\nu_n\right\}} , \ \ \forall t\in[0,T]
		\end{equation}
		and $(U^{1,n}_t)_{t=0}^T$ satisfies the pre-conditions of Theorem 1.2 in \cite{Peshkir2006} w.r.t. the filtration $\mathbb{F}$. Thus,  deduce that $\widetilde{\tau}_{n+1}$ is the minimal-optimal  stopping time of the problem
		\begin{equation}
		\text{ess}\sup_{\tau\in\mathcal{T}(\mathbb{F})}E\left(U_\tau^{1,n}|\mathcal{F}_0\right)\,,
		\end{equation} 
		and hence $\widetilde{\tau}_{n+1}\in\mathcal{T}(\mathbb{F})$. Then, showing that $\tau_{n+1}\in\mathcal{T}(\mathbb{F})$  follows by the same arguments which appear in the proof of Lemma 3.1 in \cite{Hamadene2010}. Similarly, define 
		
		\begin{equation}
		U^{2,n}_t\equiv X_t^2\textbf{1}_{\{t<\tau_n\}}+Y_t^2\textbf{1}_{\{t\leq\tau_n\}} , \ \ \forall t\in[0,T]\,,
		\end{equation}
		and notice that the current model assumptions  yield that $(U^{2,n}_t)_{t=0}^T$  satisfies the pre-conditions of Theorem 1.2 in \cite{Peshkir2006} w.r.t. $\mathbb{G}^{\tau_n}$. Then, showing that $\widetilde{\nu}_{n+1}\in\mathcal{T}(\mathbb{G}^{\tau_n})$ and hence $\nu_{n+1}\in\mathcal{T}(\mathbb{G}^{\tau_n})$ follows by the same arguments which appear in the proof of Lemma 3.1 in Hamad\'ene \textit{et al.}  \cite{Hamadene2010}. Finally, the monotonicity of the sequences $\left(\tau_n\right)_{n\geq1}$ and $\left(\nu_n\right)_{n\geq1}$ is justified by construction and it is explained in the proof of Lemma 3.1 in Hamad\'ene \textit{et al.}  \cite{Hamadene2010}.    
	\end{proof}
	
	\begin{lemma}
		For every $n\geq2$, one has:
		\begin{enumerate}
			\item \begin{equation}
			J_1(\tau,\nu_n)\leq J_1(\tau_n,\nu_n), \ \ \forall\tau\in\mathcal{T}(\mathbb{F})\,.
			\end{equation}
			
			\item \begin{equation}
			J_2(\tau_n,\nu)\leq J_2(\tau_n,\nu_n), \ \ \forall\nu\in\mathcal{T}(\mathbb{G})\,.
			\end{equation}
		\end{enumerate}
	\end{lemma}
	
	\begin{proof}
		The proof is the same as the  proof of Lemma 3.3 in \cite{Hamadene2010}.  The only thing that should be noticed (for the proof of the second part) is that $\nu\in\mathcal{T}(\mathbb{G})$ implies that $\nu\wedge\tau_n\in\mathcal{T}(\mathbb{G}^{\tau_n})$.
	\end{proof}
	\newline\newline
	Observe that the discrete-time modelling with the condition \eqref{eq: integrability} in the definition of $\mathcal{D}$ implies that a standard convergence theorem argument might be used in order to justify the following convergences:
	\begin{enumerate}
		\item $J_1(\tau_n,\nu)\rightarrow J_1(\tau^*,\nu)$ as $n\to\infty$  for every $\nu\in\mathcal{T}(\mathbb{G})$.
		
		\item $J_2(\tau,\nu_n)\rightarrow J_2(\tau,\nu^*)$ as $n\to\infty$ for every $\tau\in\mathcal{T}(\mathbb{F})$.
		
		\item For each $i=1,2$, $J_i(\tau_n,\nu_n)\rightarrow J_i(\tau^*,\nu^*)$ as $n\to\infty$.
	\end{enumerate} 
	Finally, combining these results with the previous lemmata in this subsection yields the proof of Proposition \ref{lemma: Hamadene}. $\blacksquare$ 	
	
	\subsection{Proof of Corollary \ref{cor: equivalence}}\label{subsec: corollary 1}
	Observe that the first direction is an immediate consequence of \eqref{eq: nu*}. In order to prove the other direction, assume that $\theta^*<\tau^*$. Thus, there exists $n'\geq1$ (which is a random variable) such that $\theta_{n'+1}<\tau_n$ for every $n\geq1$. In addition, according to \eqref{eq: tilde nu}, $\nu_{n'+1}=\theta_{n'+1}$, $P$-a.s. Recall that $(\tau_n)_{n\geq1}$ and $(\nu_n)_{n\geq1}$  are nonincreasing sequences. As a result, by construction, $\tau_{n'}=\tau_{n'+1}=\ldots$ and hence
	\begin{equation}
	\tau^*=\lim_{n\to\infty}\tau_n>\theta_{n'+1}\geq\nu^*, \ \ P\text{-a.s.}
	\end{equation}
	which makes \eqref{eq: iff} follows. $\blacksquare$
	
	\subsection{Proof of Lemma \ref{lemma: auxilary}}\label{subsec: lemma 1}
	The best response of Player 2 is the same in both games and hence it is sufficient to show that $\tau_0$ is a best response of Player 1 to $\nu_0$ in $\Gamma$. To this end, observe that \eqref{eq: sufficient condition} yields that $X^1_t+\epsilon>Y_t^1$ for every $0\leq t\leq T$. Therefore, since Player 1 has a priority in stopping and $\tau_0$ is a best response of Player 1 to $\nu_0$ in $\Gamma_\epsilon$, deduce that $\tau_0\leq\nu_0$, $P$-a.s. (otherwise $\tau_0\wedge\nu_0$ is a better response). Now, assume by contradiction that there exists $\tau'\in\mathcal{T}(\mathbb{F})$ such that $J_1(\tau_0,\nu_0)<J_1(\tau',\nu_0)$. Then, \eqref{eq: sufficient condition} with the priority in stopping of Player 1 implies that $\tau'\wedge\nu_0$ is also a better response of Player 1, i.e.,
	\begin{align}
	J_{1,\epsilon}(\tau'\wedge\nu_0,\nu_0)-\epsilon&=J_1(\tau'\wedge\nu_0,\nu_0)\\&\geq J_1(\tau',\nu_0)>J_1(\tau_0,\nu_0)=J_{1,\epsilon}(\tau_0,\nu_0)-\epsilon\nonumber
	\end{align}
	from which a contradiction follows. $\blacksquare$
	\begin{equation*}
	\end{equation*}
\textbf{Acknowledgement:}
The author would like to express his deep gratitude to Yan Dolinsky and Eilon Solan for insightful discussions about the problem as well as for their valuable comments on early versions of the current work. In addition, special thanks are devoted to Reviewer 1 for detecting a mathematical gap in the original manuscript and to Reviewer 2 whose suggestions lead to some major changes in the structure.

\end{document}